\documentclass[review,onefignum,onetabnum]{siamart171218}

\usepackage{bm}
\usepackage{mathtools}
\usepackage{verbatim}
\usepackage{comment}
\usepackage{multirow} 
\usepackage{xspace,soul,color}
\usepackage{xcolor}
\soulregister{\cite}7 
\soulregister{\citep}7 
\soulregister{\citet}7 
\soulregister{\ref}7 
\nolinenumbers

\usepackage{mathrsfs,amsmath,amssymb,bm}
\usepackage{lipsum}
\usepackage{amsfonts}
\usepackage{graphicx}
\usepackage{epstopdf}
\usepackage{makecell, rotating}
\usepackage{multirow}
\usepackage{algorithmic}
\usepackage{lipsum}
\usepackage{amsfonts}
\usepackage{graphicx}
\usepackage{epstopdf}
\usepackage{algorithmic}
\usepackage{tabularx}
\usepackage{hyperref}
\usepackage{cleveref}
\ifpdf
  \DeclareGraphicsExtensions{.eps,.pdf,.png,.jpg}
\else
  \DeclareGraphicsExtensions{.eps}
\fi

\newtheorem{thm}{Theorem}[section]

\newtheorem{remark}[thm]{Remark}

\newcommand*{\avint}{\mathop{\ooalign{$\int$\cr$-$}}}
\newcommand{\ds}{\displaystyle}

\newsiamremark{hypothesis}{Hypothesis}
\crefname{hypothesis}{Hypothesis}{Hypotheses}
\newsiamthm{claim}{Claim}

\headers{Nullspace-preserving high-index saddle dynamics}{Kai Jiang, Lei Zhang, Xiangcheng Zheng, Tiejun Zhou}

\title{Nullspace-preserving high-index saddle dynamics method for degenerate multiple solution problems\thanks{Submitted to the journal’s Methods and Algorithms for Scientific Computing section.
\funding{KJ and TJZ were supported by the NSFC grant (12171412), the Science and Technology Innovation Program of Hunan Province (2024RC1052). LZ was supported by the NSFC grant (12225102, T2321001, 12288101). XCZ was partially supported by the NSFC grant (12301555), the Natural Science
Foundation of Shandong Province (ZR2025QB01), the National Key R\&D Program of China (2023YFA1008903), and by the Taishan Scholars Program of Shandong Province (tsqn202306083). We acknowledge the High-Performance Computing Platform of Xiangtan University for the partial support of this work.}}}

\author{Kai Jiang\thanks{Hunan Key Laboratory for Computation and Simulation in Science and Engineering,
        Key Laboratory of Intelligent Computing and Information Processing of Ministry
        of Education, School of Mathematics and Computational Science, Xiangtan University, Xiangtan, Hunan,
        China, 411105.
        (Corresponding author. Kai Jiang, \email{kaijiang@xtu.edu.cn}).}
         \and Lei Zhang\thanks{Beijing International Center for Mathematical Research, Center for Machine Learning Research, Center for Quantitative Biology, Peking University, Beijing, China, 100871.} \and Xiangcheng Zheng\thanks{School of Mathematics, Shandong University, Jinan, Shandong, China, 250100.} \and Tiejun Zhou\footnotemark[2]}

\usepackage{amsopn}
\DeclareMathOperator{\diag}{diag}
\usepackage{amsfonts}

\makeatletter
\newcommand*{\addFileDependency}[1]{
  \typeout{(#1)}
  \@addtofilelist{#1}
  \IfFileExists{#1}{}{\typeout{No file #1.}}
}
\makeatother

\definecolor{myred}{rgb}{1,0.8,0.8}

\definecolor{mycyan}{rgb}{0.5,0.92,1.0}

\definecolor{mycyan}{rgb}{0.5,0.92,1.0}
\definecolor{mygreen}{rgb}{0.56,0.93,0.56}

\definecolor{myhl}{rgb}{1.0,0.98,0.56}
\sethlcolor{myhl}

\hypersetup{
	pdftitle={Nullspace-preserving high-index saddle dynamics method for degenerate multiple solution problems},
	pdfauthor={Kai Jiang, Lei Zhang, Xiangcheng Zheng, Tiejun Zhou}
}

\begin{document}
	\maketitle
	\begin{abstract}
	We propose the nullspace-preserving high-index saddle dynamics (NPHiSD) method for degenerating multiple solution systems in constrained and unconstrained settings. The NPHiSD efficiently locates high-index saddle points and provides parent states for downward searches of lower-index saddles, thereby constructing the solution landscape systematically. The NPHiSD method searches along multiple efficient ascent directions by excluding the nullspace, which is the key for upward searches in degenerate problems. To reduce the cost of frequent nullspace updates, the search is divided into segments, within which the ascent directions remain orthogonal to the nullspace of the initial state of each segment. A sufficient and necessary condition for characterizing the segment that admits efficient ascent directions is proved. Extensive numerical experiments for typical problems such as Lifshitz-Petrich, Gross-Pitaevskii, and Lennard-Jones models are performed to show the universality and effectiveness of the NPHiSD method.
	\end{abstract}
	
	
	
\section{Introduction}
Finding stationary points of complex systems is crucial for understanding their underlying mechanisms. The solution landscape, a comprehensive pathway map consisting of stationary points and their connections\,\cite{yin2020construction}, provides insights into the connections between multiple solutions\,\cite{wang2018two}, and has found wide applications, including locating liquid crystal defects\,\cite{shi2024multistability,wang2021modeling}, revealing the excited states and excitation mechanisms of rotational Bose-Einstein condensate\,\cite{yin2024revealing}, and studying nucleation of quasicrystals\,\cite{yin2021transition}. Due to the inherently unstable nature of saddle points, their numerical computation presents significantly greater challenges than finding minima through standard optimization methods, which has motivated extensive research efforts in this field.

For locating index-1 saddle points, numerous computational approaches have been developed and can be categorized into two main classes, that is, path-finding methods\,\cite{e2002string,e2004minimum,jonsson1998nudged} and surface-walking methods\,\cite{cui2024spring, gao2015iterative, gao2016iterative, henkelman1999dimer, yin2019high}. For computing high-index saddle points, several important approaches have been proposed. The minimax method\,\cite{chen2022improved,chen2025nehari,xie2022solving} applied the local minimax theorem to compute index-$k$ saddle points. 
The biased gradient squared descent method\,\cite{duncan2014biased} finds high-index saddle points by transforming all critical points of the original potential energy into the minima of a gradient squared landscape. In\,\cite{e2011gentlest}, the gentlest ascent dynamics method was also extended to compute index-2 saddle points, and a generalized gentlest ascent dynamics algorithm was developed in\,\cite{bofill2015some,quapp2014locating} to locate high-index saddle points on the potential energy surfaces. In\,\cite{yin2019high}, a high-index saddle dynamics (HiSD) method was recently proposed to find the high-index saddle points and construct the solution landscape.
 
Despite these significant computational advances, most existing approaches are designed for non-degenerate problems. For degenerate systems, such as crystal\,\cite{brazovskii1975phase, shi1996theory}, quasicrystals\,\cite{levine1984quasicrystals, tsai2000stable, wang2011origin}, crystal boundary\,\cite{jiang2022tilt, sutton1995interfaces}, Bose-Einstein condensate\,\cite{bao2013mathematical, bao2004computing}, cluster\,\cite{miller1997isomerization, schwerdtfeger2024years, tsai1993use} and deep neural networks \cite{daneshmand2018escaping, dauphin2014identifying}, the stationary points are often degenerate, meaning the Hessian at these points has a nullspace. This degeneracy can hinder the escape from the basin or finding higher-index saddle points. Furthermore, solution landscapes are typically constructed through successive downward searches from the highest-index saddle point to lower-index ones. When the highest-index saddle point is not given \textit{a priori}, it may need to be computed through an upward search from a lower-index saddle point. Therefore, developing efficient upward search algorithms for degenerate problems is critical for constructing their solution landscapes. Recently, a nullspace-preserving saddle search method was proposed to locate transition states from the degenerate basin by searching in the complement of the nullspace\,\cite{cui2025efficient}, which motivates the extension to high-index cases. Different from the index-1 saddle point, which is adjacent to the local minimum, searching for high-index saddle points may cross more complex energy surfaces and pass through multiple lower-index saddle points. In particular, overcoming the nullspace obstacles and designing efficient high-index algorithms for degenerate problems remain significant challenges.

In this work, we develop the nullspace-preserving high-index saddle dynamics (NPHiSD) method to study degenerate saddle problems for both constrained and unconstrained systems, which searches along multiple efficient ascent directions by excluding the nullspace. To avoid the computational cost of computing the nullspace at each state, the search is performed on several segments where the nullspace is fixed on each segment. Concerning the discrepancy between the evolutionary and fixed nullspaces, a sufficient and necessary condition for characterizing the segments that admit efficient upward directions is proved, which provides theoretical support for the algorithm. We then use the NPHiSD method to search for high-index saddle points and construct the solution landscape in three representative systems, i.e., (i) nucleation phenomenon of crystals on the Lifshitz-Petrich model; (ii) excited states in Bose-Einstein condensates on the Gross-Pitaevskii model; (iii) cluster configurations on the Lennard-Jones model. Numerical results confirm the efficiency, generality, and robustness of the proposed method.

The rest of this paper is organized as follows: Section 2 introduces notations and preliminary concepts. Section 3 presents the NPHiSD method for both constrained and unconstrained problems. Section 4 analyzes the feasibility of the segment-by-segment algorithm. Section 5 provides numerical discretization and analysis for NPHiSD. Section 6 demonstrates applications of NPHiSD in various examples, and some concluding remarks are addressed in the last section.

\section{Notations and preliminaries}
In this section, we introduce the notation and preliminary results used in the analysis.

Let $E(\phi)$ with $\phi\in \mathbb R^M$ ($0<M\in\mathbb N$) be a twice differentiable energy functional with the negative gradient $F(\phi):=-\nabla E(\phi)$ and the Hessian $H(\phi):=\nabla^2 E(\phi)$. The critical point of $E$ satisfies $\nabla E(\phi) = 0$. A commonly-used approach to distinguish different critical points is by Morse index\,\cite{stein1963morse}. The Morse index of a non-degenerate critical point $\phi$ refers to the number of negative eigenvalues of $H(\phi)$. For degenerate problems, we can extend the definition of the Morse index to depict critical points. Specifically, for a degenerate or non-degenerate critical point, the \textit{generalized Morse index} is the number of negative eigenvalues of its Hessian, and a critical point is called an index-$k$ \textit{generalized saddle point} (GSP) if its generalized Morse index is $k\, (k > 0)$. The index-1 GSP is also known as the transition state. The case $k=0$ corresponds to the \textit{generalized local minimum} (GLM). A region around a generalized critical point $\phi^{*}\in\mathbb{R}^M$ is called a \textit{generalized quadratic region} (GQR), if the Hessian $H(\phi)$ has the same number of negative eigenvalues as $H(\phi^{*})$ for any $\phi \in\mathbb{R}^M$ in this region.

Given $\phi\in\mathbb{R}^M$, we define the unstable, stable subspaces and nullspace of $H(\phi)$ as
\begin{equation*}
	\mathcal{W}^u(\phi) \!=\! {\rm span}\{v^u_1,\cdots,v^u_{l_u}\}, \,
	\mathcal{W}^s(\phi) =  {\rm span}\{v^s_{1},\cdots,v^s_{l_s}\}, \,
	\mathcal{W}^n(\phi) = {\rm span}\{v^n_{1},\cdots,v^n_{l_n}\}, 
\end{equation*}
where $\{v^u_1,\cdots,v^u_{l_u}\}$, $\{v^s_{1},\cdots,v^s_{l_s}\}$ and $\{v^n_{1},\cdots,v^n_{l_n}\}\subset \mathbb{R}^M$ are eigenvectors of $H(\phi)$
corresponding to negative, positive, and zero eigenvalues, respectively, and $l_u + l_s + l_n = M$.
According to the primary decomposition theorem\,\cite{hirsch1974differential}, $\mathbb{R}^M$ can be decomposed as
$\mathbb{R}^M = \mathcal{W}^u(\phi) \oplus \mathcal{W}^s(\phi) \oplus \mathcal{W}^n(\phi)$
where the dimensions $l_u$, $l_s$ and $l_n$ of subspaces $\mathcal{W}^u(\phi)$, $\mathcal{W}^s(\phi)$ and $\mathcal{W}^n(\phi)$ depend on the location of $\phi$ on the energy surface.
In other words, if $\phi$ is in the GQR of the GLM, $l_u = 0$. If $\phi$ in the GQR of index-$k$ GSP, $l_u = k$. If $\phi$ is in the GQR of the generalized local maximum, $l_s = 0$. To unify minima and saddle points, we define the index-$p$ stationary points. When $p = 0$, it represents the GLM, and when $p>0$, it represents the index-$p$ GSP.

Define the norm $\|\phi\|_2 = \sqrt{\sum_{i=1}^M \phi_i^2}$ and the inner product $\langle \phi,\psi \rangle=\phi^\top \psi$ for $\phi=(\phi_1, \phi_2, \cdots,\phi_M ),\, \psi =(\psi_1, \psi_2, \cdots, \psi_M)\in \mathbb{R}^M$. To measure the difference between two subspaces, we define the principal angles\,\cite{bjorck1973numerical,miao1992principal}.
Let $\mathcal{W}$ and $\widehat{\mathcal{W}}$ be subspaces in $\mathbb{R}^M$ with {\rm dim} $\mathcal{W} = d \leq$ {\rm dim} $\widehat{\mathcal{W}} = m$. Then the principal angles $0 \leq \theta_1 \leq \theta_2 \leq \cdots \leq \theta_d \leq \pi/2$ between $\mathcal{W}$ and $\widehat{\mathcal{W}}$
are defined by
\begin{align}
	\cos \theta_i := \frac{u_i^\top v_i}{\|u_i\|_{2}\|v_i\|_{2}}
	= \max \left\{ \frac{u^\top v }{\|u\|_{2}\|v\|_{2}} :~\begin{aligned}
			u\in \mathcal{W}, &~ u \bot u_k,\\
			v \in \widehat{\mathcal{W}}, &~ v \bot v_k,
	\end{aligned}~~~ k = 1,\cdots,i-1\right\},\notag
\end{align}
where $(u_i, v_i) \in \mathcal{W} \times \widehat{\mathcal{W}}$ for $ i =  1,\cdots, d$
are corresponding $d$ pairs of principal vectors. We also denote $\sin \Theta(\mathcal{W}, \widehat{\mathcal{W}})$ as
$
\sin \Theta(\mathcal{W}, \widehat{\mathcal{W}}) = \operatorname{diag}(\sin \theta_1, \sin \theta_2, \cdots, \sin \theta_d),
$
which is a $d \times d$ diagonal matrix\,\cite{bjorck1973numerical}. $\sin \Theta(\mathcal{W}, \widehat{\mathcal{W}})$ can measure the difference between the subspaces $\mathcal{W}$ and $\widehat{\mathcal{W}}$.

\section{Formulation of NPHiSD}
In this section, we propose the NPHiSD for locating high-index saddle points in degenerate problems. The motivations are as follows: Finding high-index saddle points is important for applications, such as providing the high-index parent state to construct the solution landscape\,\cite{yin2020construction}. When we intend to find an index-$k$ GSP from a GLM on the potential surface $E(\phi)$, one can ascend along $k$ eigenvectors corresponding to the smallest $k$ nonzero eigenvalues of the Hessian. However, in degenerate problems where the Hessian has a nullspace, these directions may involve components from the nullspace, leading to the failure of the search. Inspired by the nullspace-preserving saddle dynamics method for locating the index-1 GSPs\,\cite{cui2025efficient}, we propose the NPHiSD method for locating index-$k$ GSPs of degenerate problems, which excludes the nullspace when designing ascent directions. In subsequent contents, we consider both the unconstrained and sphere-constrained cases.

\subsection{Unconstrained NPHiSD}\label{sec31}
Following the high-index saddle dynamics proposed in\,\cite{yin2019high}, the state variable $\phi$ is updated according to the following dynamics
\begin{equation} \label{eq:update_phi}
	\beta^{-1} \dot{\phi} = \Big(I-2\sum_{i=1}^k v_iv_i^\top\Big)F(\phi),
\end{equation}
where $V = \text{span}\{v_1,v_2,\cdots,v_k\}$ is the ascent subspace with orthonormal basis $\{v_i\}_{i=1}^k$ and $\beta>0$ is a relaxation parameter. 

Considering the GQR of an index-$p$ GSP ($p\leq k$), we divide the search process into multiple segments. In each segment, we denote the initial state as $\bar{\phi}$ whose corresponding Hessian $H(\bar{\phi})$ has a nullspace denoted as $\mathcal{W}^{n}(\bar{\phi}) =\text{span}\{\bar{v}^n_1, \cdots,\bar{v}^n_{l_n}\}$. Now we focus our attention on the segment containing $\bar\phi$, where we determine $V$ in the complement of $\mathcal{W}^{n}(\bar{\phi})$ to obtain an efficient ascent subspace (the determination of such a segment will be discussed in the next section)
\begin{align}
	\min_{v_i} \langle v_i,H(\phi)v_i\rangle, \quad s.t. \quad \begin{cases}
		&\langle v_i,v_j \rangle = \delta_{ij},\, j = 1,2,\cdots,i,\\
            & v_i \bot \mathcal{W}^{n}(\bar{\phi}),\, i = 1, \cdots, k.\\
	\end{cases} \label{eq:min-v}
\end{align}
In principle, one can solve the minimization problem \cref{eq:min-v} at each time instant to support the evolution of the dynamics of $\phi$, which is computationally expensive. Instead, following\,\cite{yin2019high}, we update $\{v_i\}_{i=1}^k$ in \cref{eq:update_phi} through a dynamical approach coupled with
\begin{align}\label{eq:update_v}
	\gamma^{-1}\dot{v}_i & = - \Big(I- v_iv_i^{\top} - \sum_{j=1}^{i-1} 2v_jv_j^{\top} -\sum_{l=1}^{l_n}\bar{v}^n_{l}(\bar{v}_{l}^n)^{\top}\Big) H(\phi)v_i,\quad i = 1,2,\cdots,k, \notag
\end{align}
where $\gamma>0$ is a relaxation parameter.
We combine the above derivations to obtain the NPHiSD
\begin{align}\label{eq:NPHiSD}
\left\{
\begin{aligned}
\beta^{-1} \dot{\phi} &= \Big(I-2\sum_{i=1}^k v_iv_i^\top\Big)F(\phi),\\
\gamma^{-1}\dot{v}_i &= - \Big(I- v_iv_i^{\top} - \sum_{j=1}^{i-1} 2v_jv_j^{\top} - \sum_{l=1}^{l_n}\bar{v}^n_{l}(\bar{v}_{l}^n)^{\top}\Big) H(\phi)v_i,\quad i = 1,2,\cdots,k.
\end{aligned}
\right.
\end{align}
\begin{remark}
Indeed, when we determine the first $p$ ascent directions, that is, $\{v_1,\dots,v_p\}$, the nullspace is naturally excluded since the dynamics lies within the GQR of an index-$p$ GSP. Thus it is not necessary to involve $\sum_{l=1}^{l_n}\bar{v}_{l}\bar{v}_{l}^{\top} H(\phi)v_i$ in \cref{eq:NPHiSD}. Nevertheless, we keep this term in the dynamics for the sake of preserving structures, as shown in the following lemma.
\end{remark}

\begin{lemma}\label{jy1}
If the initial values  $\{v_{i}^{(0)}\}_{i=1}^k$ of the NPHiSD \cref{eq:NPHiSD} satisfy 
\begin{align}\label{eq:ncons}
\langle v_{i}^{(0)},v_{j}^{(0)} \rangle = \delta_{ij},~~ 1\leq i,j \leq k;~~\langle v_{i}^{(0)},\bar v^n_l \rangle=0,~~1\leq i\leq k,~~1\leq l\leq l_n.
\end{align}
Then these constraints are preserved for any $t>0$
\begin{align}
\langle v_{i}^{(t)},v_{j}^{(t)} \rangle = \delta_{ij},~~ 1\leq i,j \leq k;~~\langle v_{i}^{(t)},\bar v^n_l \rangle=0,~~1\leq i\leq k,~~1\leq l \leq l_n.\notag
\end{align}
\end{lemma}

\begin{proof}
Define a vector $X(t)$ with entries $\langle v_i^{(t)}, \bar v_l^n\rangle$ for $1\leq i\leq k$ and $1\leq l\leq l_n$ and $\langle v_i^{(t)}, v_j^{(t)}\rangle-\delta_{ij}$ for $1\leq i,j\leq k$. We aim to show that $X$ satisfies a homogeneous linear ordinary differential equation such that $X(0)=0$ (due to \cref{eq:ncons}) implies $X(t)\equiv 0$ for any $t>0$, which yields the claim. In particular, it is equivalent to show that the derivative of each entry of $X$ is a linear combination of the entries of $X$.

For $1\leq i\leq k$ and $1\leq l\leq l_n$, we apply the orthonormality of $\{\bar v_l^n\}_{l=1}^{l_n}$ to get
	\begin{align*}
		 \frac{d}{dt}\langle v_i^{(t)}, \bar v_l^n\rangle = & \langle \frac{d}{dt}v_i^{(t)}, \bar v_l^n\rangle=- \Big(-\langle\bar v_l^n, v_i^{(t)}\rangle (v_i^{(t)})^{\top} - \sum_{j=1}^{i-1} 2\langle\bar v_l^n,v_j^{(t)}\rangle (v_j^{(t)})^{\top} \Big) H(\phi)v_i^{(t)}.
	\end{align*}
	For $m>i$, we have
	{\footnotesize
    \begin{align*}
		& \frac{d}{dt}\langle v_{i}^{(t)},v_{m}^{(t)} \rangle =\langle \frac{d v_{i}^{(t)}}{dt},v_m^{(t)} \rangle+\langle v_{i},\frac{d v_{m}^{(t)}}{dt}\rangle\\
         &=- \Big((v_m^{(t)})^\top- \langle v_m^{(t)},v_i^{(t)}\rangle (v_i^{(t)})^{\top} - \sum_{j=1}^{i-1} 2\langle v_m^{(t)},v_j^{(t)}\rangle (v_j^{(t)})^{\top}  \\
         &~- \sum_{l=1}^{l_n}\langle v_m^{(t)},\bar{v}^n_{l}\rangle(\bar{v}_{l}^n)^{\top}\Big) H(\phi)v_i^{(t)}- \Big((v_i^{(t)})^\top- \langle v_i^{(t)},v_m^{(t)}\rangle (v_m^{(t)})^{\top} \\
         &~- \sum_{j=1}^{m-1} 2\langle v_i^{(t)},v_j^{(t)}\rangle (v_j^{(t)})^{\top}
         - \sum_{l=1}^{l_n}\langle v_i^{(t)},\bar{v}^n_{l}\rangle (\bar{v}_{l}^n)^{\top}\Big) H(\phi)v_m^{(t)}\\
         &=2(v_m^{(t)})^\top H(\phi)v_i^{(t)}(\langle v_i^{(t)},v_i^{(t)}\rangle-1) \notag \\
         &~+ \Big( \langle v_m^{(t)},v_i^{(t)}\rangle (v_i^{(t)})^{\top} + \sum_{j=1}^{i-1} 2\langle v_m^{(t)},v_j^{(t)}\rangle (v_j^{(t)})^{\top} + \sum_{l=1}^{l_n}\langle v_m^{(t)},\bar{v}^n_{l}\rangle(\bar{v}_{l}^n)^{\top}\Big) H(\phi)v_i^{(t)}\\
         &~+ \Big( \langle v_i^{(t)},v_m^{(t)}\rangle (v_m^{(t)})^{\top} + \sum_{j=1,j\neq i}^{m-1} 2\langle v_i^{(t)},v_j^{(t)}\rangle (v_j^{(t)})^{\top} + \sum_{l=1}^{l_n}\langle v_i^{(t)},\bar{v}^n_{l}\rangle(\bar{v}_{l}^n)^{\top}\Big) H(\phi)v_m^{(t)}.
	\end{align*}}
	The derivative of $\langle v_{i}^{(t)},v_{m}^{(t)} \rangle$ for $m\leq i$ can be evaluated similarly. The above derivations indicate that the derivative of each entry of $X$ is a linear combination of the entries of $X$, which completes the proof.
\end{proof}

\subsection{Sphere-constrained NPHiSD}
For problems such as locating excited states of the Bose-Einstein condensate, we need to search high-index GSPs on a sphere constraint $\|\phi\|^2 =1$. To accommodate such a constraint in the NPHiSD method, we define the Riemannian gradient as ${\rm grad}\, E(\phi) = \mathcal{P}_{\phi} F(\phi)$
where $\mathcal{P}_{\phi}$ is the projection operator in the tangent space $T(\phi) = \{ \psi: \langle  \psi, \phi \rangle = 0\}$ defined as
\begin{align}\label{eq:projection}
    \mathcal{P}_{\phi}\psi = \psi-\langle \psi, \phi\rangle \phi=(I-\phi\phi^\top)\psi.
\end{align}

Similar to the unconstrained case, to find the index-$k$ GSP, we select $k$ eigenvectors corresponding to the smallest $k$ eigenvalues of the Riemannian Hessian. The Riemannian Hessian operator is defined as 
\begin{align}
{\rm Hess}\,E(\phi)[v] = \mathcal{P}_{\phi}(\partial_{v}{\rm grad}\,E(\phi)), \notag
\end{align}
for $v \in T(\phi)$.

For the degenerate problems, the Riemannian Hessian has the nullspace $\mathcal{W}^{n}(\phi) \subset T(\phi)$. Similar to the unconstrained case, we focus on the GQR of an index-$p$ stationary point. For the segment containing $\bar\phi$, we replace $\mathcal{W}^{n}(\phi)$ by $\mathcal{W}^{n}(\bar{\phi})$ and search upward along $V = {\rm span}\{v_1, v_2, \cdots, v_k\}$ with the constraint $\phi \bot \mathcal{W}^{n}(\bar{\phi})$. The dynamics of searching for an index-$k$ GSP can be written as
\begin{equation*}
\beta^{-1} \dot{\phi} = (I - 2\sum_{i=1}^k v_iv_i^{\top} -\sum_{l=1}^{l_n}\bar v^n_l(\bar v^n_l)^\top){\rm grad}\, E(\phi).
\end{equation*}

The ascent space $V$ can be solved by 
\begin{align}
	\min_{v_i} \langle v_i,{\rm Hess}\,E(\phi)[v_i]\rangle, \quad s.t. \quad \left\{\begin{aligned}
		&\langle v_i,v_j \rangle = \delta_{ij},\, j =1,2,\cdots,i,\\
            &  v_i\bot \phi, v_i \bot \mathcal{W}^{n}(\bar{\phi}),\, i =1, 2, \cdots, k.
	\end{aligned} \right. \notag
\end{align}
Similar to the derivations in\,\cite{yin2019high}, we adopt the dynamics approach to update $\{v_i\}_{i=1}^k$ as follows
\begin{align*}
\gamma^{-1} \dot{v}_i =& -\bigg( I-v_iv_i^\top-2\sum_{j=1}^{i-1}v_jv_j^\top-\sum_{l=1}^{l_n}\bar v^n_l(\bar v^n_l)^\top\bigg){\rm Hess}\,E(\phi)[v_i] -\langle v_i, {\rm grad}\,E(\phi)\rangle \phi.
\end{align*}

By \cref{eq:projection}, we have ${\rm grad}\, E(\phi)=(I- \phi \phi^{\top})F(\phi) $ and ${\rm Hess}\,E(\phi)[v]
= -\mathcal{P}_{\phi}(H(\phi) v) +\langle \phi, F(\phi) \rangle v$, 
such that the dynamics of $\phi$ and $v_i$ can be simplified as
\begin{align*}
\beta^{-1} \dot{\phi} &= \bigg(I - 2\sum_{i=1}^k v_iv_i^{\top} - \sum_{l=1}^{l_n}\bar v^n_l(\bar v^n_l)^\top\bigg) (I- \phi \phi^{\top})F(\phi) \\
&= \bigg(I- \phi \phi^{\top} - 2\sum_{i=1}^k v_iv_i^{\top}- \sum_{l=1}^{l_n}\bar v^n_l(\bar v^n_l)^\top\bigg) F(\phi).
\end{align*}
and
\begin{align*}
\gamma^{-1} \dot{v}_i =&\bigg( I-v_iv_i^\top-2\sum_{j=1}^{i-1}v_jv_j^\top-\sum_{l=1}^{l_n}\bar v^n_l(\bar v^n_l)^\top\bigg)(\mathcal{P}_{\phi}(H(\phi)v_i) - \langle \phi,F(\phi)\rangle v_i) \notag \\
&+\langle v_i, F(\phi) - \phi \phi^\top F(\phi)\rangle \phi \notag  \\
=&\bigg( I-\phi \phi^\top - v_iv_i^\top-2\sum_{j=1}^{i-1}v_j v_j^\top-\sum_{l=1}^{l_n}\bar v^n_l(\bar v^n_l)^\top\bigg)H(\phi)v_i +\langle v_i, F(\phi)\rangle \phi,
\end{align*}
We combine the above derivations to obtain the sphere-constrained NPHiSD to locate an index-$k$ GSP
\begin{align}\label{sadk}
	\left\{
	\begin{aligned}
		\beta^{-1}\dot\phi = & \bigg(I -\phi \phi^\top-2\sum_{j=1}^k v_jv_j^\top -\sum_{l=1}^{l_n}\bar v^n_l(\bar v^n_l)^\top \bigg)F(\phi),\\
        \gamma^{-1}\dot v_i = & \bigg( I-\phi \phi^\top-v_iv_i^\top-2\sum_{j=1}^{i-1}v_jv_j^\top-\sum_{l=1}^{l_n}\bar v^n_l(\bar v^n_l)^\top\bigg)H(\phi)v_i\\
        & + \phi v_i^\top F(\phi),~~1\leq i\leq k.
	\end{aligned}
	\right.\end{align}
When $p=k$, it becomes the sphere-constrained HiSD proposed in\,\cite{yin2022constrained}. 

Compared to the unconstrained case, the sphere-constrained NPHiSD preserves more structures, as shown in the following lemma.
    
\begin{lemma}
		If $\beta=\gamma$ and the following initial constraints hold for the sphere-constrained NPHiSD \cref{sadk} 
       $$\|\phi^{(0)}\|_2=1,~ \langle\phi^{(0)},v_{i}^{(0)}\rangle=0,~\langle v_{i}^{(0)},v_{j}^{(0)}\rangle=\delta_{ij},~\langle \phi^{(0)},\bar{v}_{l}^{n}\rangle=0, ~\langle v_{i}^{(0)},\bar{v}_{l}^{n}\rangle=0,$$
       for $1\leq i,j\leq k$ and $1\leq l\leq l_n$, then the following properties hold for any $t>0$
		$$\|\phi^{(t)}\|_2=1,~ \langle \phi^{(t)},v_{i}^{(t)}\rangle=0,~\langle v_{i}^{(t)},v_{j}^{(t)}\rangle=\delta_{ij}, ~\langle \phi^{(t)},\bar{v}_{l}^{n}\rangle=0, ~\langle v_{i}^{(t)},\bar{v}_{l}^{n}\rangle=0.$$
	\end{lemma}
\begin{proof}
    The proof can be carried out analogously to that of \cref{jy1}, with the matrix $X$ extended to include the additional terms required to preserve the constraints in this lemma. Hence, the stated invariants hold for all $t>0$, which completes the proof.
\end{proof}

\section{Feasibility analysis}
In this section, we provide theoretical characterizations for the segment of the NPHiSD on which the nullspace can be fixed without affecting the effective ascent of the dynamics for an index-$k$ GSP (the sphere-constrained case can be studied similarly and is thus omitted). 
For clarity, we define the nullspace at $\phi$ as  $\mathcal{W}^n(\phi)$ and its nonzero subspace as $\mathcal{W}^c(\phi) = {\rm span}\{u^c_1, u^c_2, \cdots, u^c_{l_c}\}$ corresponding to  nonzero eigenvalues $\lambda^c_1 \leq \lambda^c_2\leq  \cdots \leq \lambda^c_{l_c}$. Without loss of generality, we assume $k\leq l_c$. 

Ideally, to effectively search for an index-$k$ GSP, one can ascend along the ascent subspace $\mathcal{V}^c(\phi)={\rm span}\{u^c_1,u^c_2,\cdots, u^c_k\}$ and descend along its orthogonal complement $(\mathcal{V}^c)(\phi)^{\bot}$, where the $\mathcal{V}^c(\phi)$ can be determined by minimizing the Rayleigh quotient $\langle v_i, H(\phi) v_i \rangle$ under constraints $\langle v_i, v_j \rangle = \delta_{ij}$ ($i,j = 1,2,..,k$) and $\mathcal{V}^c(\phi) \bot \mathcal{W}^n(\phi)$. 
However, due to the evolution of $\phi$, evaluating $\mathcal{W}^n(\phi)$ at each $\phi$ is expensive in practice. Thus, we instead perform the upward search segment by segment and in each segment, we adopt the ascent subspace $\mathcal{V}(\phi):={\rm span}\{v_1,v_2,\ldots,v_k\}$ by minimizing the Rayleigh quotient as above under constraints $\langle v_i, v_j \rangle = \delta_{ij}$ ($i,j = 1,2,..,k$) and $\mathcal{V}^c(\phi) \bot \mathcal{W}^n(\bar\phi)$, where $\mathcal{W}^n(\bar\phi)$ is a fixed nullspace $\mathcal{W}^n(\bar{\phi})$ at the initial state $\bar{\phi}$. Note that this is exactly the process described by \cref{eq:min-v}. In other words, we will compensate for the statements in \cref{sec31} by determining the segments mentioned therein.

As $\phi$ evolves, the nullspace $\mathcal{W}^n(\phi)$ may gradually diverge from the fixed nullspace $\mathcal{W}^n(\bar{\phi})$. Consequently, the constraint $\mathcal{V}(\phi) \bot \mathcal{W}^n(\bar\phi)$ in \cref{eq:min-v} may not fully exclude the nullspace $\mathcal{W}^n(\phi)$ from the dynamics, which may again lead to the ascent along part of the nullspace. To ensure $k$ efficient directions for upward search, we should at least ensure ${\rm dim } ( \mathcal{V}(\phi) \cap \mathcal{W}^c(\phi) )= k$, i.e., $\mathcal{V}(\phi)$ contains $k$ linearly independent directions in $\mathcal{W}^c(\phi)$. Then we remain to determine the segment where this condition holds.
For clarity, we define $C_{\Theta} = \|\sin \Theta(\mathcal{W}^n(\phi), \mathcal{W}^n( \bar{\phi} )\|_F$  to measure the difference between $\mathcal{W}^n(\phi)$ and $\mathcal{W}^n(\bar{\phi})$. To measure the difference between real and computed ascent directions, we define $\theta^c_j$ as the angle between $u^c_j$ and $v_j$ for $ j=1,2,\cdots,k$ and denote $\rho_{i} = \sum_{j=1}^{i-1}\sin^2(\theta^c_j), i= 2,3,\cdots,k$. A necessary and sufficient condition is then established in the following theorem.

\begin{theorem}\label{le:find_direction}
In the GQR of an index-$p$ GSP for $0<p\leq k$ (or GLM for $p=0$), if the dimension of the nullspace keeps unchanged on the trajectory of \cref{eq:update_phi}, the ascent subspace $\mathcal{V}(\phi) = {\rm span}\{v_1, v_2, \cdots, v_k\}$ determined by \cref{eq:min-v}  satisfies
\begin{align}
  {\rm dim} ( \mathcal{V}(\phi) \cap \mathcal{W}^c(\phi)) = k,\text{ if and only if }\,0\leq C^2_{\Theta} + \rho_{k} < 1.\notag
\end{align}
\end{theorem}

\begin{remark}
In this theorem, we assume that the dimension of the nullspace remains unchanged on the trajectory of \cref{eq:update_phi}. First, it is worth mentioning that this is a natural assumption in several circumstances. For instance, the translational or rotational symmetry may cause the existence of the nullspace, and the symmetry breaking will lead to a change in its dimension. Since the symmetry breaking is typically triggered at a valley ridge inflect point \cite{ramquet2000critical} that locates on the boundary of the GQR of a saddle point, the dimension of the nullspace may thus keep unchanged within the GQR, which is aligned with the assumption that the nullspace keeps unchanged on the trajectory of \cref{eq:update_phi} within the GQR.

Second, if the dimension of the nullspace changes within the GQR, we could detect such a change by computation-friendly indicators and then update the segment to ensure that the dimension of the nullspace remains unchanged within each segment. Specifically, the following two cases will be separately considered:
\begin{itemize}
\item If the dimension of the nullspace tends to increase, then there exists a non-zero eigenvalue that tends to be zero, and such a change can be simply detected by the evolution of the Rayleigh quotient in \cref{eq:min-v}.

\item If the dimension of the nullspace tends to decrease, we could check the value of
$
\max_{i} \, |\langle \bar{u}^n_i, H(\phi) \bar{u}^n_i \rangle| $. If this value exceeds a threshold $\varepsilon_{n}>0$, then we could update the segment.
\end{itemize}
 
\end{remark}

\begin{proof}
    The proof will be divided into two parts.
    
    \textbf{Case 1: $p=0$}. For this case, we consider the segment within the GQR of a GLM. Denote $\{\lambda_l^n\}_{l=1}^{l_n}$ and $\{\lambda_j^c\}_{j=1}^{l_c}$ as zero and non-zero eigenvalues of $H(\phi)$, respectively, with the corresponding eigenvectors  $\{u_l^n\}_{l=1}^{l_n}$ and $\{u_j^c\}_{j=1}^{l_c}$, where the entries in $\{\lambda_j^c\}_{j=1}^{l_c}$ are ordered from the smallest to the largest. Denote the matrices $\Lambda^c_i :=\diag\{\lambda^c_i,\lambda_{i+1}^c,\ldots,\lambda_{l_c}^c\}$ for $i\leq l_c$, $U^n =(u^n_1,u^n_2,\cdots,u^n_{l_n})^{\top}$ and $U_i^c =(u^c_i,u^c_{i+1},\cdots,u^c_{l_c})^{\top}$. 
Then the Hessian $H(\phi)$ can be expressed as
$H(\phi) = (U_1^c)^{\top}\Lambda_1^c U_1^c$.
Similarly, we can define the matrix $\bar{U}^n$ according to $\mathcal{W}^n(\bar{\phi})$. For $\phi$ in the neighborhood of $\bar{\phi}$, we have $U^n = \bar{U}^n + \delta U^n$ for some perturbation $\delta U^n$. By the cosine-sine decomposition\,\cite{davis1969some,golub2013matrix, paige1994history,stewart1990matrix}, there exist orthogonal matrices $U_1$ and $U_2$ such that 
\begin{align*}
\mathcal{P}_{\mathcal{W}^n(\phi)}(I- \mathcal{P}_{\mathcal{W}^n(\bar{\phi})}) = U_1\sin\Theta(\mathcal{W}^n(\phi),\mathcal{W}^n(\bar{\phi}))U^{\top}_2.
\end{align*} 

Then we intend to determine the smallest $k$ nonzero eigenvectors $v_1,v_2, \cdots, v_k $ from \cref{eq:min-v}.  As $v_1 \bot \mathcal{W}^n(\bar{\phi})$, we have $v_1 = (I- \mathcal{P}_{\mathcal{W}^n(\bar{\phi})})v_1$ such that
\begin{align*}
\|\mathcal{P}_{\mathcal{W}^n(\phi)} v_1\|_2 &= \|\mathcal{P}_{\mathcal{W}^n(\phi)}(I- \mathcal{P}_{\mathcal{W}^n(\bar{\phi})}) v_1\|_2  \\
    &\leq \|\mathcal{P}_{\mathcal{W}^n(\phi)}(I- \mathcal{P}_{\mathcal{W}^n(\bar{\phi})})\|_2 \leq \|\mathcal{P}_{\mathcal{W}^n(\phi)}(I- \mathcal{P}_{\mathcal{W}^n(\bar{\phi})})\|_F \notag \\
& = \sqrt{{\rm tr}(\langle U_1\sin\Theta(\mathcal{W}^n(\phi),\mathcal{W}^n(\bar{\phi}))U^{\top}_2, U_1\sin\Theta(\mathcal{W}^n(\phi),\mathcal{W}^n(\bar{\phi}))U^{\top}_2\rangle)} \notag \\
&= \sqrt{{\rm tr}(\sin^2 \Theta(\mathcal{W}^n(\phi),\mathcal{W}^n(\bar{\phi})))} = \|\sin \Theta(\mathcal{W}^n(\phi),\mathcal{W}^n(\bar{\phi}))\|_F =C_{\Theta},\notag
\end{align*}
where $\|\cdot\|_F$ is the Frobenius norm and ${\rm tr}(\cdot)$ represents the trace of the matrix. 
Denote
\begin{align*}
    v_1 & =  \mathcal{P}_{\mathcal{W}^c(\phi)}v_1 + \mathcal{P}_{\mathcal{W}^n(\phi)}v_1  = a_{11} u^c_1 + \cdots +  a_{1,l_c} u^c_{l_c} + b_{11} u^n_{1} + \cdots +  b_{1,l_n} u^n_{l_n}, 
\end{align*}
with $(a_{11})^2 + \cdots + (a_{1,l_c} )^2 + (b_{11})^2 + \cdots + (b_{1,l_n} )^2 =1$. When  $C_{\Theta} \leq 1$, we have $(a_{11})^2 + \cdots + (a_{1,l_c} )^2 \geq 1 - C^2_{\Theta}$ and
\begin{align}
    \langle v_1, H(\phi) v_1\rangle &= (a_{11})^2\lambda^c_{1} +(a_{12})^2\lambda^c_{2} + \cdots + (a_{1,l_c})^2\lambda^c_{l_c} \notag  \\
    &\geq ((a_{11})^2 + (a_{12})^2  + \cdots + (a_{1,l_c})^2)\lambda_1^{c}\geq (1-C_{\Theta}^{2})\lambda_1^{c}. \notag
    \end{align}
Since $v_1 \in \mathbb{R}^n$, $v_1$ can be selected in the following form to attain the above lower bound 
\begin{align}\label{jy4}
v_1 = \sqrt{1-C_{\Theta}^2} u^c_1 + \sum_{l=1}^{l_n} b_{1,l} u^n_l, \, \lambda_1 = (1 -C^2_{\Theta})\lambda^c_1, 
\end{align}
where $\sum_{l=1}^{l_n} b^2_{1,l} =C_{\Theta}$.

 Suppose the eigenpairs $(\lambda_{j}, v_{j})$ for $ j=1,2,\cdots,i-1$ are obtained from \cref{eq:min-v} with the splittings
${v}_{j} = u^c_{j} + w^c_{j}$
 for some perturbations $\{w^c_{j}\}_{j=1}^{i-1}$,
 and we intend to determine $v_i$ and $\lambda_i$. Denote 
\begin{align} 
v_i = a_{i1} u^c_1 + \cdots +  a_{i,l_c} u^c_{l_c} + b_{i1} u^n_{1} + \cdots +  b_{i,l_n} u^n_{l_n}, \label{eq: vi_present}
\end{align}
where $(a_{i1})^2 + \cdots + (a_{i,l_c} )^2 \geq 1 - C^2_{\Theta}$ by similar derivations in the case of $v_1$. 
Based on the orthonormality of $\{v_j\}_{j=1}^{i-1}$, we have
\begin{align}\label{eq:rayleigh_vi} 
   \langle v_i, H(\phi) v_i \rangle = & \sum_{j=1}^{i-1}\langle v_i, \lambda^c_j u^c_j(u^c_j)^{\top} v_i \rangle + \langle v_i, (U^c_{i})^{\top}\Lambda^c_{i} U^c_{i} v_i \rangle \\
   =& \sum_{j=1}^{i-1}\langle v_i, \lambda^c_j (v_j - w^c_j)(v_j - w^c_j)^{\top} v_i \rangle + \langle v_i, (U^c_{i})^{\top}\Lambda^c_{i} U^c_{i} v_i \rangle \notag\\
   =&  \sum_{j=1}^{i-1}\lambda_j^c( (w^c_j)^{\top} v_i)^2+ \langle v_i, (U^c_{i})^{\top}\Lambda^c_{i} U^c_{i} v_i \rangle \notag\\
  = &   \sum_{j=1}^{i-1}\lambda_j^c( \|w^c_j\|_2 \|v_i\|_2 \cos(\theta^c_j/2))^2 + \langle v_i, (U^c_{i})^{\top}\Lambda^c_{i} U^c_{i} v_i \rangle \notag\\
=&  \sum_{j=1}^{i-1}\lambda_j^c\sin^2(\theta^c_j) + \langle v_i, (U^c_{i})^{\top}\Lambda^c_{i} U^c_{i} v_i \rangle  \notag \\
    \geq& \sum_{j=1}^{i-1} \lambda_j^c\sin^2(\theta^c_j) + (1 -C^2_{\Theta} -\sum_{j=1}^{i-1}\sin^2(\theta^c_j))\lambda_i^c, \notag
\end{align}
where we used $\|w^c_j\|_2 = 2\|u_j^c\|_2 \sin(\theta^c_j/2) = 2\sin(\theta^c_j/2)$ with a simple illustration shown in \cref{fig:angles}. 
\begin{figure}[!htp]
	\centering
	\includegraphics[width=0.4\textwidth]{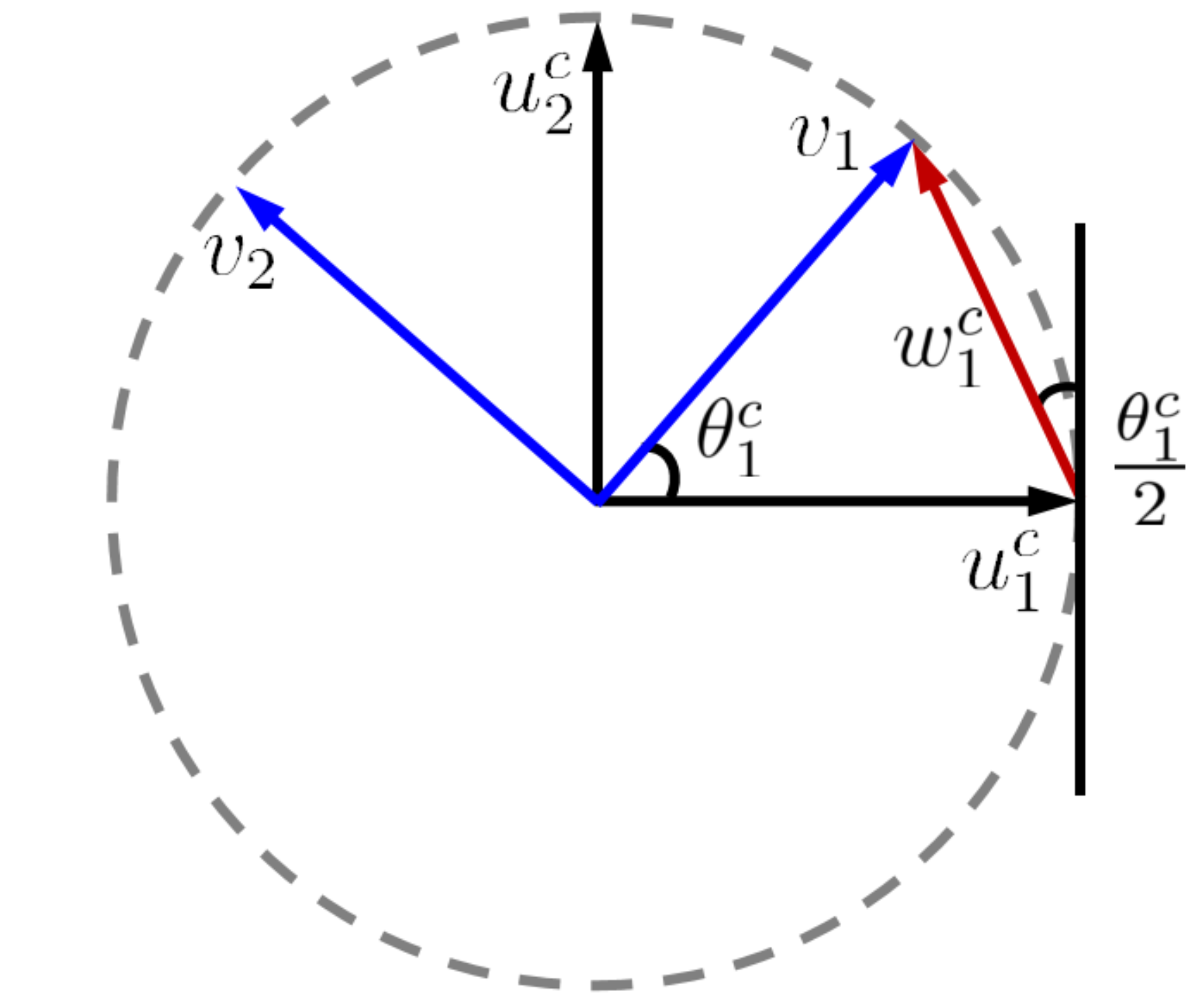}
	\caption{A simple illustration of $v_j$, $u^c_j$ and $\theta_j^c$ for $j=1$.}
	\label{fig:angles}
\end{figure}
Similar to the derivations around \cref{jy4}, we conclude that $\lambda_i$ takes the following form
\begin{align}
\lambda_i= \sum_{j=1}^{i-1} \lambda_j^c\sin^2(\theta^c_j) + (1 -C^2_{\Theta} -\sum_{j=1}^{i-1}\sin^2(\theta^c_j))\lambda_i^c, \notag
\end{align}
with the corresponding eigenvector $v_i$ 
\begin{align}
{v}_i = \sum_{j=1}^{i-1} \sin(\theta^c_j) u_j^c + \sqrt{1 -C^2_{\Theta} -\sum_{j=1}^{i-1}\sin^2(\theta^c_j)} u_i^c + \sum_{l=1}^{l_n} b_{i,l} u^n_l, \notag
\end{align}
where $\sum_{l=1}^{l_n} b^2_{i,l} =C_{\Theta}$.

Denote
$\rho_{i} = \sum_{j=1}^{i-1}\sin^2(\theta^c_j), i= 2,3,\cdots,k$ with $\rho_{1} = 0$ and
\begin{align*}
P = \left(\begin{array}{cccc}
\sqrt{1-C^2_{\Theta}-\rho_{1}} &                &        & \\
\sin(\theta_1^c)    &  \sqrt{1-C^2_{\Theta}-\rho_{2}} &   & \\
 \vdots                   &  \vdots                &  \ddots &\\
\sin(\theta_1^c)&\sin(\theta_2^c) & \cdots       &\sqrt{1-C^2_{\Theta}-\rho_{k}}
\end{array}\right),
\end{align*}
and
\begin{align*}
Q = \left(\begin{array}{cccc}
b_{11} & b_{12} & \ldots       & b_{1,l_n} \\
b_{21} & b_{22} & \ldots       & b_{2,l_n} \\
 \vdots & \vdots                  &                   & \vdots\\
b_{k1} &  b_{k2} &\ldots       & b_{k,l_n}
\end{array}\right).
\end{align*}
Then
we have 
\begin{align}
 V(\phi) = PV^c(\phi) + QU^n, \label{eq:transform_PQ}
\end{align}
where $V(\phi) = ({v}_1, {v}_2, \cdots, {v}_k )^{\top}$ and $V^c(\phi) = (u^c_1, u^c_2, \cdots, u^c_k )^{\top}$ are concatenated from basis vectors of $\mathcal{V}(\phi)$ and $\mathcal{V}^c(\phi)$, respectively.
As $V(\phi) \bot U^n$, we need ${\rm rank }(P) = k$ to ensure ${\rm rank }(V(\phi)) = k$, which in turn requires
$
0\leq C^2_{\Theta} +\rho_{k}<1$.

\textbf{Case 2: $0<p\leq k$}. For the state in the GQR of the index-$p$ GSP, the first $p$ smallest eigenvalues are negative. 
If $ 0 \leq C_{\Theta} <1$, then the maximum principal angle $ 0\leq \theta_{l_n}(\mathcal{W}^n(\phi),\mathcal{W}^n(\bar{\phi})) < \pi/2$, which implies that no non-zero vector $v \in \mathcal{W}^n(\phi)$ orthogonal to $ \mathcal{W}^n(\bar{\phi})$. Equivalently, 
$(\mathcal{W}^n(\bar{\phi}))^{\bot}  \cap \mathcal{W}^n(\phi) = \{\bm{0}\}$.
Therefore, solving \eqref{eq:min-v} can obtain $p$ orthonormal eigenvectors $\{v_1,v_2,\cdots,v_p\}$ in $(\mathcal{W}^n(\phi))^\bot$, which corresponds to $p$ negative eigenvalues of $H(\phi)$. Hence, the computation of the first $p$ eigenvectors imposes no restrictions on the choice of segment.

For $v_i$ with $ i = p+1,\cdots, k$,  the nullspace should be excluded from the dynamics. 
Similar to \cref{eq: vi_present} and \cref{eq:rayleigh_vi}, we have
\begin{align*}
   \langle v_i, H(\phi) v_i \rangle &=  \sum_{j=1}^{i-1}\langle v_i, \lambda^c_j u^c_j(u^c_j)^{\top} v_i \rangle + \langle v_i, (U^c_{i})^{\top}\Lambda^c_{i} U^c_{i} v_i \rangle  \\
   &\hspace{-0.4in} = \sum_{j=1}^{i-1} \lambda_j^c\sin^2(\theta^c_j) + \langle v_i, (U^c_{i})^{\top}\Lambda^c_{i} U^c_{i} v_i \rangle   = \sum_{j=1}^{i-1} \lambda_j^c\sin^2(\theta^c_j) + \sum_{j=i}^{l_c}(a_{i,j})^2\lambda_j   \notag  \\
   &\hspace{-0.4in} \geq \sum_{j=1}^{i-1} \lambda_j^c\sin^2(\theta^c_j) + (\sum_{j=i}^{l_c}(a_{i,j})^2)\lambda_i   \geq \sum_{j=1}^{i-1} \lambda_j^c\sin^2(\theta^c_j) + (1 -C^2_{\Theta} -\sum_{j=1}^{i-1}\sin^2(\theta^c_j))\lambda_i^c, \notag
\end{align*}
and one can directly verify that the above lower bound can be attained by selecting
\begin{align}
\lambda_i= \sum_{j=1}^{i-1} \lambda_j^c\sin^2(\theta^c_j) + (1 -C^2_{\Theta} -\sum_{j=1}^{i-1}\sin^2(\theta^c_j))\lambda_i^c, \notag
\end{align}
with the corresponding eigenvector $v_i$ 
\begin{align*}
{v}_i = \sum_{j=1}^{i-1} \sin(\theta^c_j) u_j^c + \sqrt{1 -C^2_{\Theta} -\sum_{j=1}^{i-1}\sin^2(\theta^c_j)} u_i^c + \sum_{l=1}^{l_n} b_{i,l} u^n_l,
\end{align*}
where $\sum_{l=1}^{l_n} b^2_{i,l} =C_{\Theta}$. 

Based on this expression, let $\hat{V}(\phi) = (v_{p+1},v_{p+2}, \cdots, v_{k})^{\top}$,
{\footnotesize
\begin{align*}
\hat{P} = \left(\begin{array}{ccccccc}
\sin(\theta_{1}^c) & \cdots &\sin(\theta_{p}^c) &\sqrt{1-C^2_{\Theta}-\rho_{p+1}} &                 &        & \\
\sin(\theta_{1}^c) & \cdots &\sin(\theta_{p}^c) & \sin(\theta_{p+1}^c)    &\sqrt{1-C^2_{\Theta}-\rho_{p+2}} &        & \\
 \vdots                   &                  & \vdots &  \vdots & \vdots & \ddots\\
\sin(\theta_{1}^c) & \cdots &\sin(\theta_{p}^c) &  \sin(\theta_{p+1}^c)& \sin(\theta_{p+2}^c)                 & \cdots       &\sqrt{1-C^2_{\Theta}-\rho_{k}}
\end{array}\right),
\end{align*}}
and 
\begin{align*}
\hat{Q} = \left(\begin{array}{cccc}
b_{p+1,1} &  b_{p+1,2}               &  \cdots      & b_{p+1,l_n} \\
b_{p+2,1} &  b_{p+2,2}               &  \cdots      & b_{p+2,l_n} \\
 \vdots                   &    \vdots              &  & \vdots\\
b_{k,1} &  b_{k,2}               &  \cdots      & b_{k,l_n}
\end{array}\right),
\end{align*}
such that $
    \hat{V}(\phi) = \hat{P}V^c(\phi) + \hat{Q} U^n$, where $V^c(\phi)$ is defined similarly as that in \cref{eq:transform_PQ}.
To ensure ${\rm dim}(\mathcal{V}(\phi) \cap \mathcal{W}^c(\phi)) = k$, we require ${\rm rank } (\hat{P}) = k-p$, which is valid if and only if $
0\leq C^2_{\Theta} +\rho_{k}<1$.
Thus, we complete the proof. \end{proof}

\begin{remark}
$C^2_{\Theta} +\rho_{k} = 1$ implies that $\hat{v}^c_k$ can be linearly represented by $\hat{v}^c_1, \hat{v}^c_2$, $\cdots$, $\hat{v}^c_{k-1}$. This fails to guarantee $k$
independent ascent directions. Consequently, the nullspace must be updated, and a new segment must be initiated. Furthermore, direct evaluation of $C_{\Theta}$ and $\rho_k$ is computationally expensive. Alternatively, we update the segment whenever at least one eigenvalue approaches zero (or smaller than a given small threshold).
\end{remark}

\section{Numerical approximation}\label{sec5}
We consider numerical methods for NPHiSD and their error estimates\,\cite{li2018near,nesterov2013introductory}. The sphere-constrained NPHiSD can be discretized and analyzed similarly by combining the techniques in\,\cite{zhang2023discretization}, and we thus omit the details.

For a time interval $[0,T]$ for some $T>0$, we define a uniform partition $0=t_0<t_1<\cdots<t_N=T$ for some $N>0$ with the step size $\tau=T/N$.
We use the Euler scheme to discretize the time derivatives of the NPHiSD \cref{eq:NPHiSD}
\begin{align}\label{eq:discrete_NPHiSD}
\left\{
	\begin{aligned}
	\phi^{(t_{n})} =&\phi^{(t_{n-1})}+\tau\beta\bigg(I -2\sum_{j=1}^k v_j^{(t_{n-1})}(v_j^{(t_{n-1})})^\top \bigg)F(\phi^{(t_{n-1})})+O(\tau^2) ,\\
	 v_i^{(t_{n})}=&v_i^{(t_{n-1})}-\tau\gamma\bigg( I-v_i^{(t_{n-1})}(v_i^{(t_{n-1})})^\top-2\sum_{j=1}^{i-1}v_j^{(t_{n-1})}(v_j^{(t_{n-1})})^\top\\
	 &-\sum_{l=1}^{l_n}\bar v^n_l (\bar v^n_l)^\top\bigg)H(\phi^{(t_{n-1})})v_i^{(t_{n-1})}+O(\tau^2),~~1\leq i\leq k,
    \end{aligned}\right.
\end{align}
equipped with initial conditions
\begin{align}\label{eq:initial_conditions}
	\phi^{(t_0)}=\phi^{(0)},~~v_i^{(t_0)}=v_{i}^{(0)},~~1\leq i\leq k,
\end{align}
satisfying the condition \cref{eq:ncons}.

Similar to the numerical methods of HiSD in the literature, such as\,\cite{luo2024semi,zhang2022error}, both explicit and semi-implicit schemes can be considered for the NPHiSD.

\underline{\textbf{Explicit scheme}}
We simply drop the truncation errors in \cref{eq:discrete_NPHiSD} to obtain the explicit Euler scheme of the NPHiSD \cref{eq:NPHiSD} with the approximations $\{\phi_n,v_{i,n}\}_{n=1,i=1}^{N,k}$ to $\{\phi^{(t_n)},v_i^{(t_n)}\}_{n=1,i=1}^{N,k}$
\begin{align}\label{eq:explicit_NPHiSD}
	\left\{
	\begin{aligned}
		\phi_{n} =&\phi_{n-1}+\tau\beta\bigg(I -2\sum_{j=1}^k v_{j,n-1}v_{j,n-1}^\top \bigg)F(\phi_{n-1}),\\
		\tilde v_{i,n}=&v_{i,n-1}-\tau\gamma\bigg( I-v_{i,n-1}v_{i,n-1}^\top\\
	 &-2\sum_{j=1}^{i-1}v_{j,n-1}v_{j,n-1}^\top-\sum_{l=1}^{l_n}\bar v^n_l(\bar v^n_l)^\top\bigg)H(\phi_{n-1})v_{i,n-1},~~1\leq i\leq k,\\
		v_{i,n}= &\frac{1}{Y_{i,n}}\bigg(\ds\tilde v_{i,n}-\sum_{j=1}^{i-1}(\tilde v_{i,n}^\top v_{j,n})v_{j,n}\bigg),~~1\leq i\leq k.
	\end{aligned}
	\right.
\end{align}
for $1\leq n\leq N$ and
\begin{align}
\begin{aligned}
	\ds Y_{i,n}:\,=\bigg\|\tilde v_{i,n}-\sum_{j=1}^{i-1}(\tilde v_{i,n}^\top v_{j,n})v_{j,n}\bigg\|_2=\bigg(\|\tilde v_{i,n}\|_2^2-\sum_{j=1}^{i-1}(\tilde v_{i,n}^\top v_{j,n})^2\bigg)^{1/2}, \notag
\end{aligned}
\end{align}
equipped with the initial conditions \cref{eq:initial_conditions} satisfying the constraints in \cref{eq:ncons}.

The third equation of \cref{eq:explicit_NPHiSD} corresponds to the Gram-Schmidt orthogonalization procedure, ensuring the set $\{v_{i,n}\}_{i=1}^k$ remains orthonormal, consistent with the continuous problem \cref{eq:NPHiSD}. Note that such orthonormal preservation inherently determines the formulation of the dynamics of $\phi$ and is critical to ensure the convergence of the algorithm to the target saddle point, referencing\,\cite{miao2025construction} for numerical demonstrations. 

The explicit scheme automatically preserves the nullspace as shown in the following lemma.
\begin{lemma}
  Under the constraints for initial values in \cref{eq:ncons}, the numerical solutions to the explicit scheme \cref{eq:explicit_NPHiSD} satisfies 
  \begin{align}
\langle v_{i,n},\bar v^n_l \rangle=0,~~1\leq i\leq k,~~1\leq l \leq l_n,~~0\leq n\leq N. \notag
\end{align}
\end{lemma}
\begin{proof}
The case $n=0$ is ensured by \cref{eq:ncons}, and then we can prove by mathematical induction.
If $v_{i,n-1}^\top \bar v^n_l=0$ for $1\leq i\leq k$ and $1\leq l\leq l_n$, one can multiply $(\bar v^n_l)^\top$ on both sides of the second equation of \cref{eq:explicit_NPHiSD} and apply the orthonormality of $\{\bar v^n_l\}_{l=1}^{l_n}$ to find $(\bar v^n_l)^\top \tilde v_{i,n}=0$ for $1\leq i\leq k$ and $1\leq l\leq l_n$.  As $v_{i,n}$ is indeed a linear combination of $\{\tilde v_{j,n}\}_{j=1}^i$, we immediately obtain
\begin{align}\label{or2}
	(\bar v^n_l)^\top  v_{i,n}=0\text{ for }1\leq i\leq k\text{ and }1\leq l\leq l_n.
\end{align}
Hence, we apply mathematical induction to conclude that \cref{or2} holds for $1\leq n\leq N$ so that the numerical scheme \cref{eq:explicit_NPHiSD} automatically preserves the nullspace as in the continuous case, which completes the proof.
\end{proof}

Using the relation $\bar v_l^\top  v_{i,n}=0$ for $1\leq i\leq k$ and $1\leq l\leq l_n$, which implies that the additional factor $\sum_{l=1}^{l_n} \bar v_l\bar v_l^\top$ compared to HiSD has rare impacts in numerical analysis, one can follow exact the same procedure as the numerical analysis for the explicit scheme of HiSD in\,\cite{zhang2022error} to prove that, under certain regularity assumptions of the model, the numerical solutions to the explicit scheme \cref{eq:explicit_NPHiSD} are bounded over $t\in [0, T]$ and satisfy the following error estimate
\begin{align}
	\|\phi^{(t_n)}-\phi_n\|_2+\sum_{i=1}^k\|v_i^{(t_n)}-v_{i,n}\|_2\leq Q\tau,~~1\leq n\leq N,\notag
\end{align}
where $Q$ is a generic positive constant that is independent from $\tau$, $n$ and $N$.

\underline{\textbf{Semi-implicit scheme}}
In practice, the force $F(\phi)$ typically consists of a linear and a nonlinear part, that is,
$
F(\phi) = \mathcal{L} \phi + \mathcal{N}(\phi)$,
with the corresponding Hessian $H(\phi) = -\mathcal{L} -\mathcal{N}'(\phi)$. We replace the term $\tau\beta F(\phi_{n-1}) $ on the right-hand side of the explicit scheme of $\phi_n$ in \cref{eq:explicit_NPHiSD} by $\tau\beta (\mathcal{L}\phi_n + \mathcal{N}(\phi_{n-1}))$, and replace the term $-\tau\gamma H(\phi_{n-1})v_{i,n-1}$ on the right-hand side of the explicit scheme of $\tilde v_{i,n}$ in \cref{eq:explicit_NPHiSD} by the semi-implicit formulation $\tau\gamma (\mathcal{L}\tilde v_{i,n} + \mathcal{N}'(\phi_{n}) v_{i,n-1})$.
This yields a semi-implicit scheme for the NPHiSD \cref{eq:NPHiSD}, which produces approximations $\{\phi_n,v_{i,n}\}_{n=1,i=1}^{N,k}$ to $\{\phi(t_n),v_i(t_n)\}_{n=1,i=1}^{N,k}$
\begin{align}\label{eq:semi_implicit_NPHiSD}
	\left\{
	\begin{aligned}
		\phi_{n} =& \phi_{n-1}+\tau\beta (\mathcal L\phi_n+\mathcal N(\phi_{n-1}))-2\tau\beta \sum_{j=1}^k v_{j,n-1}v_{j,n-1}^\top F(\phi_{n-1}),\\
		\tilde v_{i,n}= & v_{i,n-1}+\tau\gamma (\mathcal L\tilde v_{i,n}+\mathcal N'(\phi_{n}) v_{i,n-1})+\tau\gamma\bigg( v_{i,n-1}v_{i,n-1}^\top\\
		& +2\sum_{j=1}^{i-1}v_{j,n-1}v_{j,n-1}^\top+  \sum_{l=1}^{l_n}\bar v^n_l(\bar v_l^n)^\top\bigg)H(\phi_{n})v_{i,n-1},~~1\leq i\leq k,\\
		\hat v_{i,n}=& \tilde v_{i,n}-\sum_{l=1}^{l_n} \tilde v_{i,n}^\top \bar v^n_l \bar v^n_l,~~1\leq i\leq k,\\
		v_{i,n}= & \frac{1}{Y_{i,n}}\bigg(\ds\hat v_{i,n}-\sum_{j=1}^{i-1}(\hat v_{i,n}^\top v_{j,n})v_{j,n}\bigg),~~1\leq i\leq k.
	\end{aligned}
	\right.
\end{align}
Compared with the explicit scheme, a further process of $\{\tilde v_{i,n}\}_{i=1}^k$ before Gram-Schmidt process is adopted to ensure their orthogonality with the nullspace, which in turn ensures the orthogonality between $\text{span}\{v_{i,n}\}_{i=1}^k$ and the nullspace $\text{span}\{\bar v^n_l\}_{l=1}^{l_n}$. Although the formulation of the semi-implicit scheme is more complicated than the explicit scheme, one can follow procedures similar to the numerical analysis for the semi-implicit scheme of HiSD in\,\cite{luo2024semi} to prove the boundedness of the numerical solutions to the semi-implicit scheme \cref{eq:semi_implicit_NPHiSD} and the error estimate
\begin{align} \label{eq:error_phi_v}
	\|\phi^{(t_n)}-\phi_n\|_2+\sum_{i=1}^k\|v_i^{(t_n)}-v_{i,n}\|_2\leq Q\tau,~~1\leq n\leq N,
\end{align}
where $Q$ is a generic positive constant that is independent of $\tau$, $n$ and $N$.

\section{Numerical experiments}
We employ the (sphere-constrained) NPHiSD for studying, e.g., the nucleation phenomenon of crystals, excited states in Bose-Einstein condensates, and Lennard-Jones clusters to show its effectiveness and universality.
In practical computations, the dynamics are terminated once
$\|F(\phi)\|_{\ell^2}<$ 1e-7 and adopt the Barzilai-Borwein step size\,\cite{yin2019high}, except for testing the temporal convergence rate, where a uniform step size is used. For PDE models in the first two examples, the semi-implicit scheme is employed, while the explicit scheme is adopted in the finite-dimensional problem in the third example. All methods are implemented in MATLAB R2021b and run on the same workstation, 12th Gen Intel (R) Core (TM) i7-12700KF, 16 GB memory, under Linux.

\subsection{Nucleation of crystals}
We study the nucleation phenomenon of crystals using the Lifshitz-Petrich model, which serves as an effective framework for describing the phases and phase transitions of quasiperiodic systems, including bifrequent Faraday waves and soft‑matter quasicrystals\,\cite{lifshitz1997theoretical,jiang2015stability,lifshitz2007soft}. The Lifshitz-Petrich model introduces a scalar order parameter $\phi(r)$ to represent the density profile in a domain $\Omega$, with the corresponding free energy 
\begin{align}\label{LP_energy}
	E(\phi) = \frac{1}{|\Omega|} \int_{\Omega} \left\{\frac{1}{2}[(1+\Delta)(q^2+\Delta)\phi]^2 + \frac{\varepsilon}{2} \phi^2 - \frac{\alpha}{3}\phi^3 + \frac{1}{4}\phi^4 \right\} {\rm d}\bm{r},
\end{align}
where $\varepsilon$ is a temperature-like controlling parameter and $\alpha$ characterizes the intensity of the three-body interaction. Here, the parameters $1$ and $q$ represent two characteristic length scales required to stabilize quasicrystals. Furthermore, we impose the mean-zero condition of the order parameter on the Lifshitz-Petrich systems to ensure the mass conservation $\int_{\Omega} \phi(r)dr = 0$, which comes from the definition of the order parameter, i.e., the deviation from average density.

\begin{figure}[!htp]
	\centering
	\includegraphics[width=0.5\textwidth]{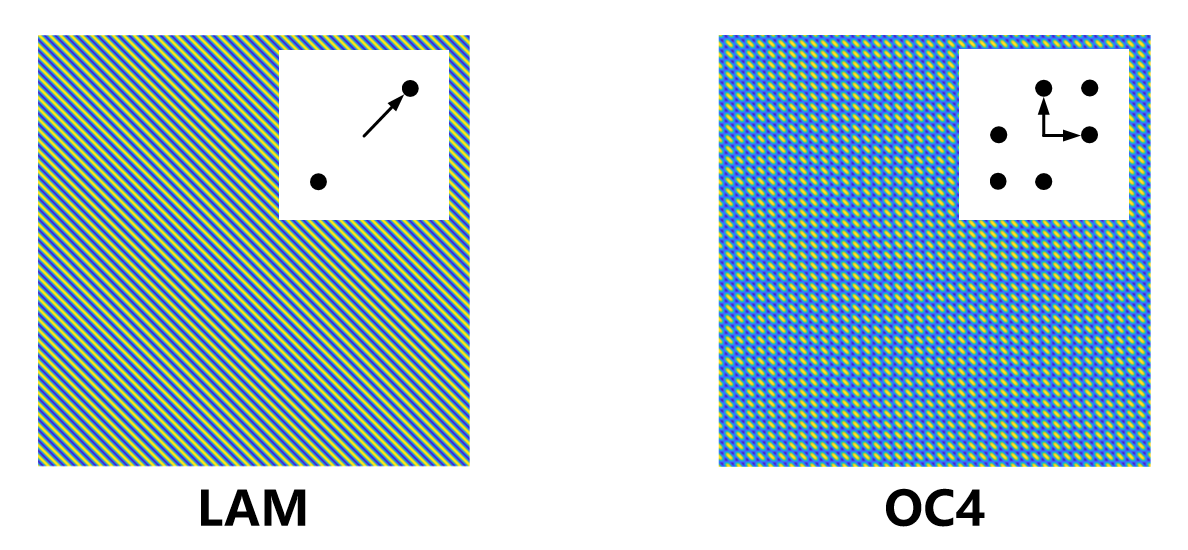}
	\caption{Stable states of the Lifshitz-Petrich model, where the subgraphs show their spectral points in the reciprocal space (the arrows present the primitive reciprocal vectors).}
	\label{fig:crystal}
\end{figure}

We choose $\Omega= [-82\pi, 82\pi)^2$, $q = 2\cos(\pi/4)$, $\varepsilon= 0.03$, $\alpha= 0.1$, and employ the Fourier pseudo-spectral method for spatial discretization with 800 spatial discretization points in each dimension \cite{zhang2008efficient}. Two stable states, lamellar (LAM) and oblique 4-fold crystal (OC4), in \cref{fig:crystal} are computed via the gradient flow approach. Their eigenvalues are computed by the locally optimal block preconditioned conjugate gradient (LOBPCG) method\,\cite{knyazev2001toward}, as indicated in \cref{tab:eigenvalues}. Since the LAM and OC4 are periodic on 1- and 2-dimensional spaces, respectively, their corresponding Hessians have 1- and 2-dimensional nullspaces, indicating that the NPHiSD method is required for the upward search. 
\begin{table}[!htp]
	\centering
	\caption{The smallest four eigenvalues of the Hessian corresponding to stable ordered structures in the Lifshitz-Petrich model. Eigenvalues less than 1e-9 are considered as zero eigenvalues. }
	\begin{tabularx}{9cm}{c|cccc}
		\hline
		States &  &Eigenvalues  & &  \\
		\hline
		LAM  & 1.002e-10 & 1.418e-6 & 1.422e-6 & 2.286e-5 \\
		OC4 & 5.696e-10  & 6.465e-10  & 4.650e-4  & 4.650e-4\\
		\hline
	\end{tabularx}
    \label{tab:eigenvalues}
\end{table}

We then construct the solution landscape of the Lifshitz-Petrich energy \cref{LP_energy} as shown in \cref{fig:solution_landscape}. Starting from the metastable state OC4, we implement the NPHiSD to locate an index-2 GSP that has two nuclei of the LAM phase with E = -1.417169e-4. From this parent state, we subsequently perform downward searches using the standard HiSD method to identify index‑1 GSPs, which correspond to transition states or critical nuclei connecting two GLMs via the minimum energy pathway.

\begin{figure}[!htp]
	\centering
	\includegraphics[width=0.7\textwidth]{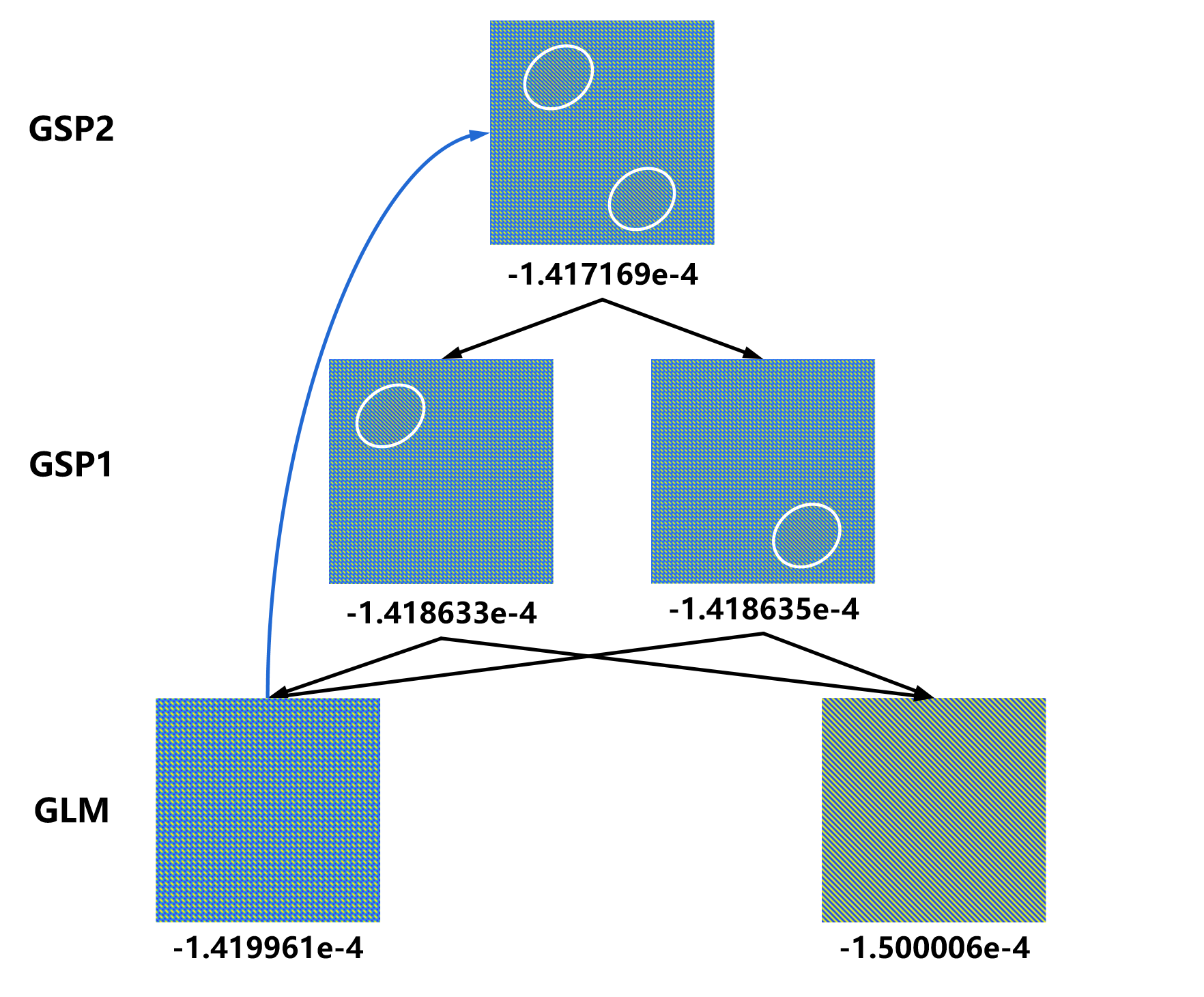}
	\caption{Solution landscape of crystal (with energy for each solution) where the GSPi (i = 1, 2) represents the index-i GSP. The critical nuclei are marked by white curves. The blue arrow represents the upward search by NPHiSD, and the block arrows represent the downward searches by HiSD (similar notations will be used subsequently).}
	\label{fig:solution_landscape}
\end{figure}

Finally, we employ the phase transition from OC4 to GSP2 in \cref{fig:solution_landscape} to test the temporal convergence orders for the semi-implicit scheme of the NPHiSD. Let $\beta = 0.2$, $\gamma = 0.3$, $T = 0.8$.
At $t = 0.8$, we adopt the numerical solution solved under the fine time step size $\tau = T/1024$ as the \textit{reference solution}. The $L^2$-errors defined in \cref{eq:error_phi_v} at $t= 0.8$ and the corresponding orders are presented in \cref{tab:error_order}, confirming the first‑order accuracy established in \cref{sec5}.

\begin{table}[ht]
	\centering
	\caption{Errors and convergence orders of semi-implicit scheme for NPHiSD method.} 
	\begin{tabular}{ccccccc}
		\hline
		\hline
		$\tau$ &0.8 & 0.4 & 0.2 & 0.1 & 0.05 & 0.025 \\ 
		\hline 
		$L^2$-error &4.815e-4 & 2.383e-4 & 1.187e-4 & 5.912e-5 & 2.932e-5 & 1.442e-5    \\
		\hline
		Order & - & 1.014 & 1.004 & 1.006 & 1.011 & 1.023 \\
		\hline
	\end{tabular}
	\label{tab:error_order}
\end{table}

\subsection{Excited states in Bose-Einstein condensate}
We aim to identify degenerate excited states in the Bose-Einstein condensate with the help of NPHiSD. Consider the dimensionless Gross-Pitaevskii energy\,\cite{bao2013mathematical,bao2004computing,liu2023constrained,yin2024revealing}
\begin{align}
	E(\phi) = \frac{1}{\Omega}\int_{\Omega} \left[\frac{1}{2} \lvert\nabla \phi(\bm{x})\rvert^2 + V(\bm{x})\lvert\phi(\bm{x})\rvert^2 + \frac{\eta}{2}\lvert\phi(\bm{x})\rvert^4\right]d\bm{x}, \label{Gross-Pitaevskii_energy}
\end{align}
on the unit sphere $\phi(\bm{x}) \in \mathcal{M} = \{\varphi(\bm{x}) \in L^2(\mathbb{R}^2) : E(\varphi) < \infty, \avint \|\varphi\|^2 d\bm{x} =1\}$. Here, $\Omega$ is a bounded domain of the system with volume $|\Omega|$, $\phi$ is the complex-valued wave function of the Bose-Einstein condensate defined on $\mathbb{R}^2$, and $\|\phi\|^2$ represents the
particle density. $\eta$ characterizes the interaction
rate and $V(\bm{x})$ is chosen as a quadrilateral trapping potential
\begin{align}
	V(\bm{x})= \frac{\omega}{2}\lvert\bm{x}\rvert^2 + \sum_{\bm{k} \in \mathcal{K}} e^{i\bm{k}^{\top}\bm{x}}, \quad \mathcal{K} = \{\bm{k}\in \mathbb{Z}^2 \mid (1, 0), (0, 1), (-1, 0), (0, -1)\}. \notag
\end{align}
Due to the rotational invariance, the stationary points of the Gross-Pitaevskii energy are usually degenerate \cite{bao2013mathematical,yin2024revealing}.

\begin{figure}[!htp]
	\centering
	\includegraphics[width=0.8\textwidth]{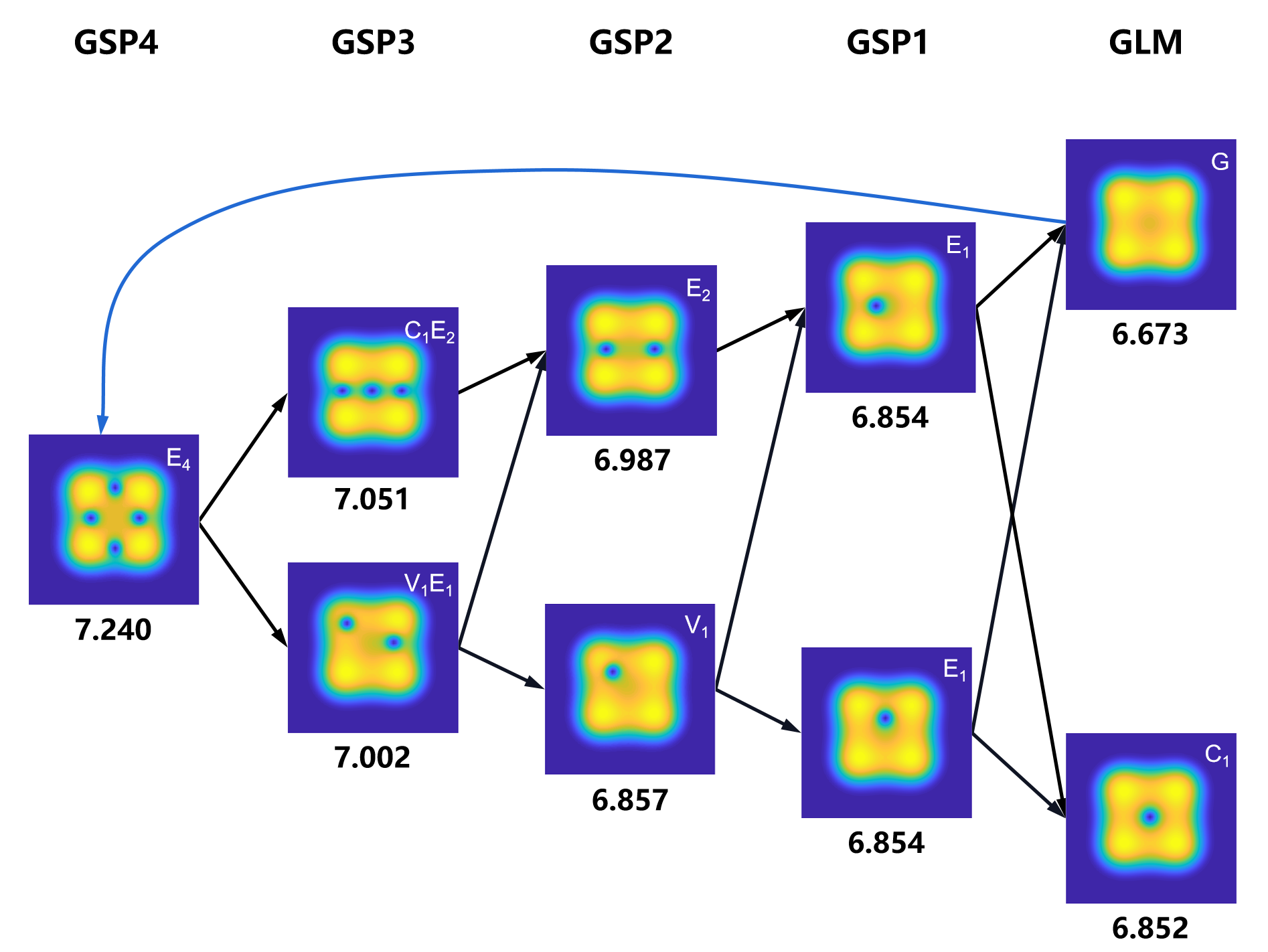}
	\caption{Vortex landscape of Bose-Einstein condensate (with energy for each solution). Here, $C$, $E$, $V$, and $G$ represent the centering, edge, and vertex vortices and the ground state, respectively, and the subscript indicates the number of vortices.}
	\label{fig:vortex_landscape}
\end{figure}
 Let $\Omega=[-2\pi,2\pi]^2$, $\omega = 1$ and $\eta = 300$, we apply the Fourier pseudo-spectral method for spatial discretization with $ 64$ degrees of freedom in each dimension. Then we construct the solution landscape of the Bose-Einstein condensate based on the Gross-Pitaevskii energy \cref{Gross-Pitaevskii_energy} with the help of the NPHiSD method shown in \cref{fig:vortex_landscape}. We first compute the ground state $G$ by the corresponding eigenvalue problems with the Newton--Noda iteration\,\cite{du2022newton}, and then search for an index-$4$ GSP by the sphere-constrained NPHiSD method. The solution landscape is then completed by downward searches by the standard HiSD method.

From \cref{fig:vortex_landscape}, we identify a vortex-adding path, i.e., $G \rightarrow E_1 \rightarrow E_2 \rightarrow C_1E_2 \rightarrow E_4$. As the number of vortices increases, the energy of the Bose-Einstein condensate system also rises, likely because the vortices enhance the kinetic energy. Furthermore, for the Bose-Einstein condensate system confined within a quadrilateral geometric structure, the position of vortices also affects the system energy. Among states with the same number of vortices, those with symmetric vortex distributions exhibit lower energy, such as $V_1E_1$ and $E_2$.

\subsection{Lennard-Jones cluster}
We consider the Lennard-Jones cluster, which provides a realistic simulation of microparticle interactions based on the Lennard-Jones potential\,\cite{schwerdtfeger2024years}
\begin{align}
\nu(r) = \frac{1}{r^{12}} - \frac{2}{r^6},\notag
\end{align}
where $r$ is the distance between two particles. The energy surface can be characterized by $V_{LJ}(\bm{r}^N)$,
\begin{align}
V_{LJ}(\bm{r}^N) = \sum_{i=1}^{N-1} \sum_{j=i+1}^N \nu(|\bm{r}_i - \bm{r}_j|), \notag
\end{align}
where $\bm{r}^N = \{\bm{r}_1, \bm{r}_2, \cdots, \bm{r}_N\}$ represents the positions of all $N$ particles. Because of rotational and translational invariances, stationary points of the Lennard–Jones energy are degenerate, and their Hessians possess nontrivial nullspaces.

We focus on the cluster system with 7 particles in 3-dimensional space\,\cite{tsai1993use,miller1997isomerization}.
We first obtain the GLM with $E = -15.5$ using the gradient descent method starting from a random configuration. By the LOBPCG method, we find that the Hessian at this GLM has a three-dimensional nullspace. We then perform an upward search by the NPHiSD method to find an index-3 GSP with the energy $E = -13.8$, and search downward for lower-index GSPs and GLMs to construct the solution landscape. As shown in \cref{fig:clusterlandscape}, multiple cluster configurations with distinct energies are identified. 

\begin{figure}[!htp]
	\centering
	\includegraphics[width=0.9\textwidth]{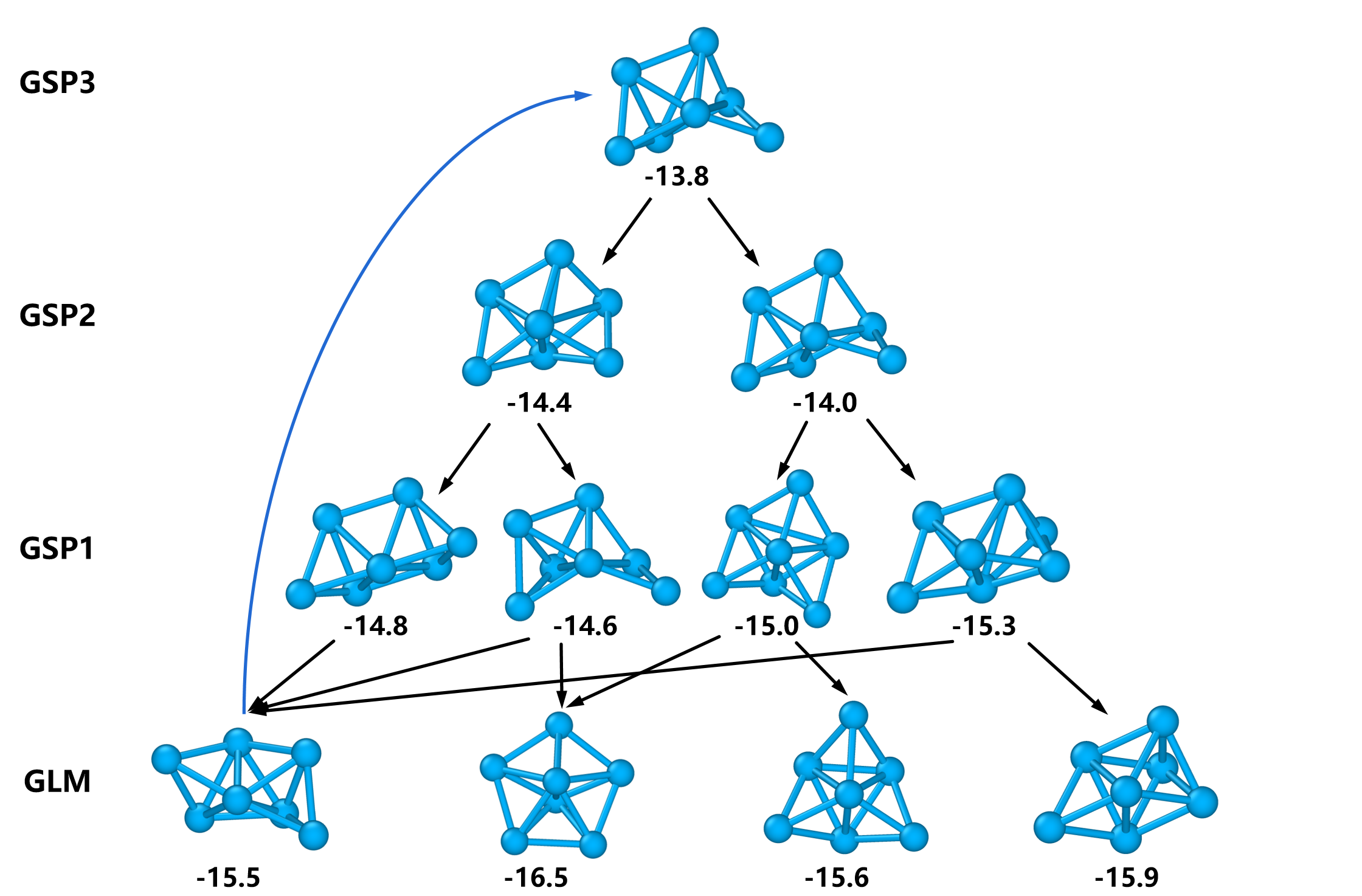}
	\caption{Solution landscape of Lennard-Jones clusters with seven particles (with energy for each state).}
	\label{fig:clusterlandscape}
\end{figure}

\section{Concluding remarks}
In this work, we develop the NPHiSD method for both constrained and unconstrained degenerate problems. To reduce the computational cost of computing the nullspace at each state, the search is performed on several segments where the nullspace is fixed on each segment. A sufficient and necessary condition for characterizing the segment that admits efficient upward directions is proved, which provides theoretical support for the algorithm. We then use the NPHiSD method to search for high-index saddle points and construct the solution landscape for field-based and particle-based systems, which confirms the efficiency, universality, and robustness of the NPHiSD method for degenerate problems. Apart from the considered models, the NPHiSD method may be further applied to other systems involving degenerate states, such as quantum systems, superconductivity, and skyrmions, and we will investigate these interesting topics in the near future.

    
\bibliographystyle{siamplain}

\end{document}